\documentclass[pamm,a4paper,fleqn]{w-art}
\usepackage{times,cite,w-thm}
\usepackage[T1]{fontenc}
\usepackage[utf8]{inputenc}
\usepackage{graphicx}
\usepackage{amsmath}
\usepackage{amssymb}
\usepackage{amsfonts}
\usepackage{color}
\usepackage{bm}
\usepackage{mathtools}
\usepackage{subcaption}
\usepackage{textcomp}
\usepackage{algpseudocode}
\usepackage{multicol}
\usepackage{algorithm}
\usepackage[labelformat=simple]{subcaption}
\usepackage{afterpage}

\def\c{\gamma}

\def\e{\varepsilon}

\def\l{\lambda}

\def\p{\pi}

\newcommand{\V}[1]{ \mathbf{#1} }    

\newcommand{\Vo}{\V{0}}

\newcommand{\Vb}{\V{b}}

\newcommand{\Vf}{\V{f}}

\newcommand{\Vr}{\V{r}}

\newcommand{\Vu}{\V{u}}

\newcommand{\Vx}{\V{x}}

\newcommand{\Vrho}{\boldsymbol{\rho}}
\newcommand{\Vxi}{\boldsymbol{\xi}}

\newcommand{\Vzero}{\boldsymbol{0}}

\newcommand{\M}[1]{ \mathbf{#1} }  

\newcommand{\MA}{\M{A}}
\newcommand{\MB}{\M{B}}

\newcommand{\MD}{\M{D}}
\newcommand{\ME}{\M{E}}
\newcommand{\MF}{\M{F}}

\newcommand{\MI}{\M{I}}
\newcommand{\MJ}{\M{J}}
\newcommand{\MK}{\M{K}}

\newcommand{\MP}{\M{P}}

\newcommand{\MU}{\M{U}}
\newcommand{\MV}{\M{V}}

\newcommand{\MSigma}{\boldsymbol{\Sigma}}

\newcommand{\cK}{{\mathcal K}} 

\newcommand{\Cmn}[2]{\mathbb{C}^{#1 \times #2}}
\def\Cp{\mathbb{C}}

\def\2nm#1{\|#1\|_2}

\def\1nm#1{\|#1\|_1}

\newcommand{\ars}[1]{\left[ \begin{array}{#1}}
\newcommand{\are}{\end{array} \right] }
\newcommand{\oars}[1]{\begin{array}{#1}}
\newcommand{\oare}{\end{array}}
\newcommand{\rars}[1]{\left( \begin{array}{#1}}
\newcommand{\rare}{\end{array} \right) }

\newcommand{\eqs}{\begin{eqnarray}}
\newcommand{\eqe}{\end{eqnarray}}
\newcommand{\eqsn}{\begin{eqnarray*}}
\newcommand{\eqen}{\end{eqnarray*}}

\def\defs{\begin{definition}}
\def\defe{\end{definition}}
\def\teos{\begin{theorem}}
\def\teoe{\end{theorem}}
\def\prfs{\begin{proof}}
\def\prfe{\end{proof}}
\def\exas{\begin{exampl}}
\def\exae{\end{exampl}}
\def\excs{\begin{exercise}}
\def\exce{\end{exercise}}
\def\cors{\begin{corollary}}
\def\core{\end{corollary}}

\newcommand{\ens}{\begin{enumerate}}
\newcommand{\ene}{\end{enumerate}}

\newcommand{\its}{\begin{itemize}}
\newcommand{\ite}{\end{itemize}}

\newcommand{\des}{\begin{description}}
\newcommand{\dee}{\end{description}}

\begin{document}

\newcommand{\AKC}[1]{{\color{blue} Arielle Notes: #1}}
\newcommand{\EdS}[1]{{\color{green} EdS: #1}}

\newcommand{\MPE}[1]{{\color{red} #1}}


\TitleLanguage[EN]
\title[The short title]{Analysis of GMRES for Low-Rank and Small-Norm Perturbations of the Identity Matrix}

\author{\firstname{Arielle K.} \lastname{Carr}\inst{1,}%
\footnote{Corresponding author: \ElectronicMail{arg318@lehigh.edu}}} 
\address[\inst{1}]{\CountryCode[DE]Department of Computer Science and Engineering, Lehigh University, Bethlehem, Pennsylvania, USA }
\author{\firstname{Eric} \lastname{de Sturler}\inst{2,}%
     \footnote{sturler@vt.edu}}
\address[\inst{2}]{\CountryCode[FR]Department of Mathematics, Virginia Tech, Blacksburg, Virginia, USA}
\author{\firstname{Mark} \lastname{Embree}\inst{2,}%
    \footnote{embree@vt.edu}}
\AbstractLanguage[EN]
\begin{abstract}
In many applications, linear systems arise where the
coefficient matrix takes the special form 
$\MI + \MK + \ME$, where $\MI$ is the identity matrix of
dimension $n$, ${\rm rank}(\MK) = p \ll n$, and 
$\|\ME\| \leq \epsilon < 1$. GMRES convergence rates for
linear systems with coefficient matrices of the forms 
$\MI + \MK$ and $\MI + \ME$ are guaranteed by well-known
theory, but only relatively weak convergence bounds
specific to matrices of the form $\MI + \MK + \ME$
currently exist. In this paper, we explore the
convergence properties of linear systems with such coefficient
matrices by considering the pseudospectrum of
$\MI + \MK$.\ \ We derive a bound for the GMRES residual in
terms of $\epsilon$ when approximately solving the linear
system $(\MI + \MK + \ME)\Vx = \Vb$ and identify the eigenvalues
of $\MI + \MK$ that are sensitive to perturbation. In
particular, while a clustered spectrum away from the origin
is often a good indicator of fast GMRES convergence, that
convergence may be slow when some of those eigenvalues are
ill-conditioned. We show there can be at most $2p$
eigenvalues of $\MI + \MK$ that are sensitive to
small perturbations. We present numerical results when using GMRES to solve a sequence of linear systems of the form 
$(\MI + \MK_j + \ME_j)\Vx_j = \Vb_j$
 that arise from the application of Broyden's method to solve a nonlinear partial differential equation.
\end{abstract}
\maketitle                   

\section{Introduction}
 
In this paper, we consider the convergence of iterative solvers for linear systems with coefficient matrices of the form
\eqs\label{eq:IKE}
\MI + \MK + \ME \in \Cmn{n}{n} ,
\eqe
where $\MI$ is the identity matrix,
$\MK$ is {\it low rank} (i.e., ${\rm rank}{(\MK)} = p \ll n$),
and $\ME$ is {\it small-in-norm} (i.e., 
$\2nm{\ME}= \epsilon < 1$). Sequences of matrices of
the special form (\ref{eq:IKE})
arise in many applications. 
As an example, we consider a preconditioned nonlinear partial differential equation (PDE) solved using Broyden's method \cite{Broy1965,Kelley2003}. 
Linear systems of this type also arise in methods where a highly effective initial preconditioner gets updated using a sequence of rank-one or low-rank updates \cite{AhuStu-etal_2011,BergBru_2006}.
We study the convergence of GMRES \cite{SaadSchu86} applied to linear systems of the form $(\MI+\MK+\ME)\Vx = \Vb$. 
At iteration $m$, GMRES approximates the solution to $\MA\Vx=\Vb$ with the estimate $\Vx_m \in \cK^m(\MA; \Vr_0) = {\rm span}\{\Vr_0, \MA\Vr_0,\dots,\MA^{m-1}\Vr_0\}$ that minimizes the 2-norm of the residual $\Vr_m = \Vb-\MA\Vx_m$.   
To start, consider the following two well-known convergence bounds with $\MI$, $\MK$, and $\ME$ defined as above (for the proofs, see \cite{CampIpse96,GmatPhil08} and
\cite{EhrDeu1997}, respectively).
\teos\label{teo:smallNorm}
Let $\MA = \MI + \ME$, 
Then $\2nm{\Vr_m} \leq \epsilon^m\2nm{\Vr_0}$.
\teoe
\teos\label{teo:lowRank}
Let $\MA = \MI + \MK$. 
Then GMRES will converge in 
at most $p+1$ iterations: $\2nm{\Vr_{p+1}}=0$.
\teoe

In practice, $n$ is often very large.  
Given a prescribed convergence tolerance $\tau>0$, we consider
convergence to be fast when $\2nm{\Vr_m}\leq \tau$ for $m \ll n$.
Clearly,  the bound in Theorem \ref{teo:smallNorm}
guarantees faster convergence for smaller $\epsilon$,
though $\epsilon$ need not be very small to ensure rapid GMRES convergence (consider, for example, $\epsilon = 1/3$). 
Further, as we are interested in cases where
$p \ll n$, Theorem \ref{teo:lowRank} guarantees very
fast convergence. To our knowledge, only
relatively weak bounds are currently known for the
convergence of GMRES when solving linear
systems with coefficient matrices of the form 
(\ref{eq:IKE}).
In \cite{KelKev2004}, for ${\rm rank}(\MK)=p$, it is shown that there exists $\mathcal{C} > 0$ such that (for 
GMRES) 
$\2nm{\Vr_{(p+1)M}} \leq \mathcal{C}^M\2nm{\ME}^M\2nm{\Vr_0}$,
for $M>0$. In particular, one can write the (scaled) minimal polynomial for $\MI+\MK$ in the form $\mu(z) = 1 + \sum_{k = 1}^{p+1}\beta_k z^k$, and take
\begin{equation} \label{eq:KKQ}
\mathcal{C} = \displaystyle\sum_{k = 1}^{q+1} k|\beta_k|\Big\Vert\MI+\MK\Big\Vert_2^{k-1}+O\left(\Big\Vert\ME\Big\Vert_2^2\right).
\end{equation}
Further, the empirical observation is made that
when $\2nm{\ME} \ll \tau$, GMRES tends 
to converge in $p+1$ iterations~\cite{KelKev2004}.  However, this condition severely
limits the size of the perturbation to $\MI+\MK$ we can
consider. 

This paper is outlined as follows. In Section \ref{sec:pseudo}, we consider the
pseudospectrum of $\MI+\MK$ and derive a straightforward bound for the GMRES residual
in terms of $\epsilon$.  In Section \ref{sec:Sources}, we identify those
eigenvalues of $\MI+\MK$ that are potentially sensitive to perturbation (i.e., the
introduction of $\ME$). 
In Section \ref{sec:results}, we study a sequence of linear systems with coefficient matrices of type~(\ref{eq:IKE}) arising from the solution of a nonlinear PDE using Broyden's method and GMRES. Finally,
in Section \ref{sec:conclusion}, we provide conclusions and future work.  All plots of pseudospectra were computed using
EigTool~\cite{Wri2002}. 
 
\section{Pseudospectral GMRES Bounds for Perturbed Matrices}\label{sec:pseudo}
In our analysis, rather than directly considering coefficient matrices of the form (\ref{eq:IKE}), we examine how $\2nm \ME$ affects the 
eigenvalues of $\MA+\ME$, where 
\eqs\label{eq:Amat}
\MA =\MI + \MK.
\eqe
Using spectral perturbation theory applied
to the resolvents of a general coefficient matrix, $\MA$, and
the perturbation, $\MA + \ME$, it is shown in \cite{SifuEmbr13} that the norm of the GMRES residual
when solving the unperturbed linear system (i.e., $\MA\Vx=\Vb$) can only increase by at most $O(\epsilon)$, regardless of
the magnitude of the change to the eigenvalues of $\MA$, provided that $\|\ME\|_2 < 1/\|\MA^{-1}\|_2$~\cite{SifuEmbr13}.  Here we specialize this result for the case when the coefficient matrix takes the particular form~(\ref{eq:Amat}).

The $\delta$-pseudospectrum of the matrix $\MA\in\Cmn nn$ can be defined in the equivalent ways (see, e.g., \cite[chap.~2]{TE05}):
\eqs\label{eq:dPseudo}
\sigma_\delta(\MA) = \left\{z\in \Cp \mid \|(z\MI - \MA)^{-1}\|_2 > 1/\delta\right\}
\ =\ 
\left\{z \in \sigma(\MA+\ME) \mbox{ for some $\ME\in\Cmn nn$ with $\|\ME\|_2 < \delta$}\right\},
\eqe
where $\sigma(\cdot)$ denotes the spectrum (set of eigenvalues).
Note that $\sigma(\MA)$ is contained in
$\sigma_\delta(\MA)$ for all $\delta > 0$ and the boundary
$\partial\sigma_\delta(\MA)$ encloses a region containing
all eigenvalues of $\MA$.\ \  When $\delta > \epsilon$, we can further conclude that
$\partial\sigma_\delta(\MA)$ encloses the region containing
all eigenvalues of $\MA+\ME$ \cite{SifuEmbr13}. Let $\delta_0>0$ denote the smallest value for which $\partial\sigma_{\delta_0}(\MA)$ passes through the origin; then $\sigma_{\delta_0}(\MA)$ is the largest $\delta$-pseudospectrum of $\MA$ that does not contain the origin.
Substituting (\ref{eq:Amat}) into (\ref{eq:dPseudo}) gives
\eqsn
\sigma_\delta(\MI +\MK) = \left\{z \in \Cp ~ \mid ~ \|((z-1)\MI-\MK)^{-1}\|_2 > 1/\delta\right\},
\eqen
and hence $\sigma_\delta(\MI+\MK) = 1 + \sigma_\delta(\MK)$.  
In Section \ref{sec:Sources} we identify precisely those eigenvalues of $\MI+\MK$ that are sensitive to small-in-norm perturbations.

Let $\boldsymbol{\rho}_m$ denote the residual at the $m$th iteration of (full) GMRES applied to $(\MI+\MK)\Vx = \Vb$,
and $\Vr_m$ denote the analogous quantity for GMRES applied to $(\MI+\MK+\ME)\Vx = \Vb$ (both with the same $\Vx_0$).
Applying \cite[Corollary 2.2]{SifuEmbr13} to $\MA = \MI+\MK$ gives, for all $\delta > \epsilon = \|\ME\|_2$, 
\eqs \label{eq:sifbound}
    \2nm{\Vr_m} \leq \2nm{{\boldsymbol{\rho}}_m} + \epsilon\, C_m(\delta),
\eqe 
where
\eqs \label{eq:Cm}
   C_m(\delta) = \displaystyle\left(\frac{L_\delta\2nm{\Vb}}{\pi\delta^2}\right)\sup_{z\in \sigma_\delta(\MA)}|\psi_m(z)|.
\eqe 
Here $L_\delta$ denotes the arc length of
$\partial\sigma_\delta(\MI+\MK)$, and
$\psi_m(z)$ denotes the residual polynomial at the $m$th  iteration of GMRES applied to
$(\MI+\MK)\Vx=\Vb$. 
Theorem~\ref{teo:lowRank} guarantees that 
$\2nm{{\boldsymbol{\rho}}_{p+1}} = 0$, and hence for any $m\ge p+1$ the bound in (\ref{eq:sifbound}) becomes
\eqs \label{eq:bound}
    \2nm{\Vr_{m}} < \epsilon\, C_{m}(\delta), \qquad m\ge p+1.
\eqe 
This bound provides a concrete alternative to the asymptotic expression in~(\ref{eq:KKQ}), and formalizes an empirical observation made in \cite{KelKev2004}: When 
$\epsilon < \tau/C_{p+1}(\delta)$, we can guarantee our
perturbed system will converge in $p+1$ steps. Of course this last observation still limits
the size of $\2nm{\ME}$ we can consider.  However, the analysis in~(\ref{eq:bound}) provides a mechanism for bounding GMRES convergence for $\MI+\MK+\ME$ at iterations \emph{beyond $m=p+1$}.  At iterations $m$ beyond the point at which GMRES has exactly converged for the $\MI+\MK$ system, the residual polynomial $\psi_m(z)$ can be taken to be any polynomial for which $\psi_m(0)=1$ and $\psi_m(\MI+\MK)\Vrho_0 = \Vzero$ (an idea employed by Ymbert to analyze a different situation in which GMRES converges quickly for the unperturbed system~\cite{Ymb11}).  Here we simply bound $\2nm{\Vr_{(p+1)M}}$ using $\psi_{(p+1)M}(z) = \psi_{p+1}(z)^M$, and the bound~\eqref{eq:bound} gives convergence when $\epsilon < \tau/C_{(p+1)M}(\delta)$.

As a simple example, consider a matrix of the form (\ref{eq:Amat}), with $n = 25$, rank$(\MK) = 2$, and three eigenvalues: $\lambda_1=1$ of multiplicity~23, and simple eigenvalues $\lambda_2 \approx 1.25$ and $\lambda_3 \approx 12.25$.
Figure
\ref{fig:Cp1eig} shows pseudospectra of this $\MI+\MK$, and Figure \ref{fig:Cp1bound} shows $\2nm{\Vr_m}$,
$\2nm{\boldsymbol{\rho}_m}$ (for 100 random perturbations $\ME$), and the upper bound
(\ref{eq:bound}) on $\2nm{\Vr_m}$ for different $\delta$.
We took perturbations of sizes $\epsilon = 10^{-3.5}$ and
$\epsilon = 10^{-5}$, and used several values of 
$\delta \in (\epsilon,1/\2nm{\MA^{-1}})$. The right-hand side $\Vb$ is fixed throughout, generated by MATLAB's {\tt randn} command and normalized to be a unit vector.
The unperturbed system converges to tolerance $\tau = 10^{-10}$
in three iterations, as guaranteed by Theorem
\ref{teo:lowRank}.  The perturbation causes a slight delay, with GMRES converging to $\tau = 10^{-10}$ in four iterations for $\2nm{\ME} = 10^{-5}$, and five iterations for $\2nm{\ME} = 10^{-3.5}$.

We estimate the bound~\eqref{eq:bound} for $\delta \le 10^{-2}$ by computing the three components of $\sigma_\delta(\MI+\MK)$ in EigTool~\cite{Wri2002} (reducing the axes to appropriate scale as $\delta$ decreases).  From this data MATLAB's contour plotting routine approximates the boundary $\partial \sigma_\delta(\MI+\MK)$ as discrete points that determine three closed curves, from which we estimate the boundary length $L_\delta$.  The $\sup$ term in~\eqref{eq:bound} is estimated by taking the maximum value of $|\psi_m(z)|$ for $z$ among the contour points that approximate  $\partial \sigma_\delta(\MI+\MK)$.

This process results in the estimated bounds shown in Figure~\ref{fig:Cp1bound}.  For reference, $C_3(10^{-3}) \approx 256.7$, $C_6(10^{-3}) \approx 2.52$, and $C_9(10^{-3}) \approx 0.025$.  These values are multiplied by $\epsilon=\|\ME\|_2$, so, for example, when $\epsilon = 10^{-5}$ and $\delta=10^{-3}$, \eqref{eq:bound} ensures convergence to $\|\Vr_3\|_2<3\cdot 10^{-3}$, $\|\Vr_6\|_2<3\cdot 10^{-5}$, and $\|\Vr_9\|_2 < 3\cdot 10^{-7}$.  While  Figure~\ref{fig:Cp1bound} suggests that these bounds can be somewhat pessimistic, they will capture eventual convergence.
Note that while $\delta$ must be chosen such that 
$\delta > \epsilon$, \eqref{eq:Cm}~itself does not depend on $\epsilon$. Thus very small perturbations of
(\ref{eq:Amat}) will give small values of $\epsilon\, C_m$.  

\begin{figure}[!t]
\begin{center}
\includegraphics[height=2
in]{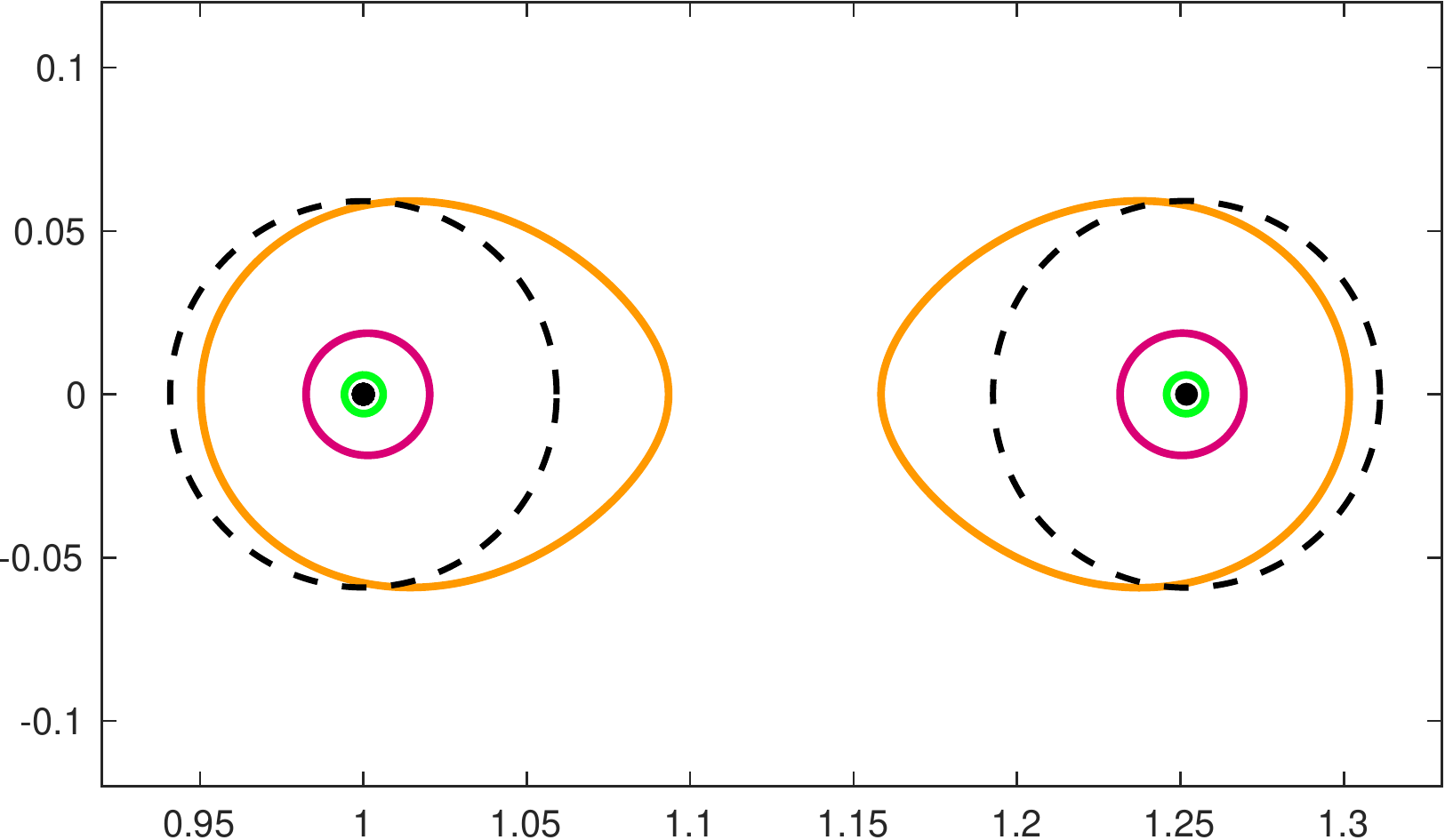}
\qquad
\includegraphics[height=2in]{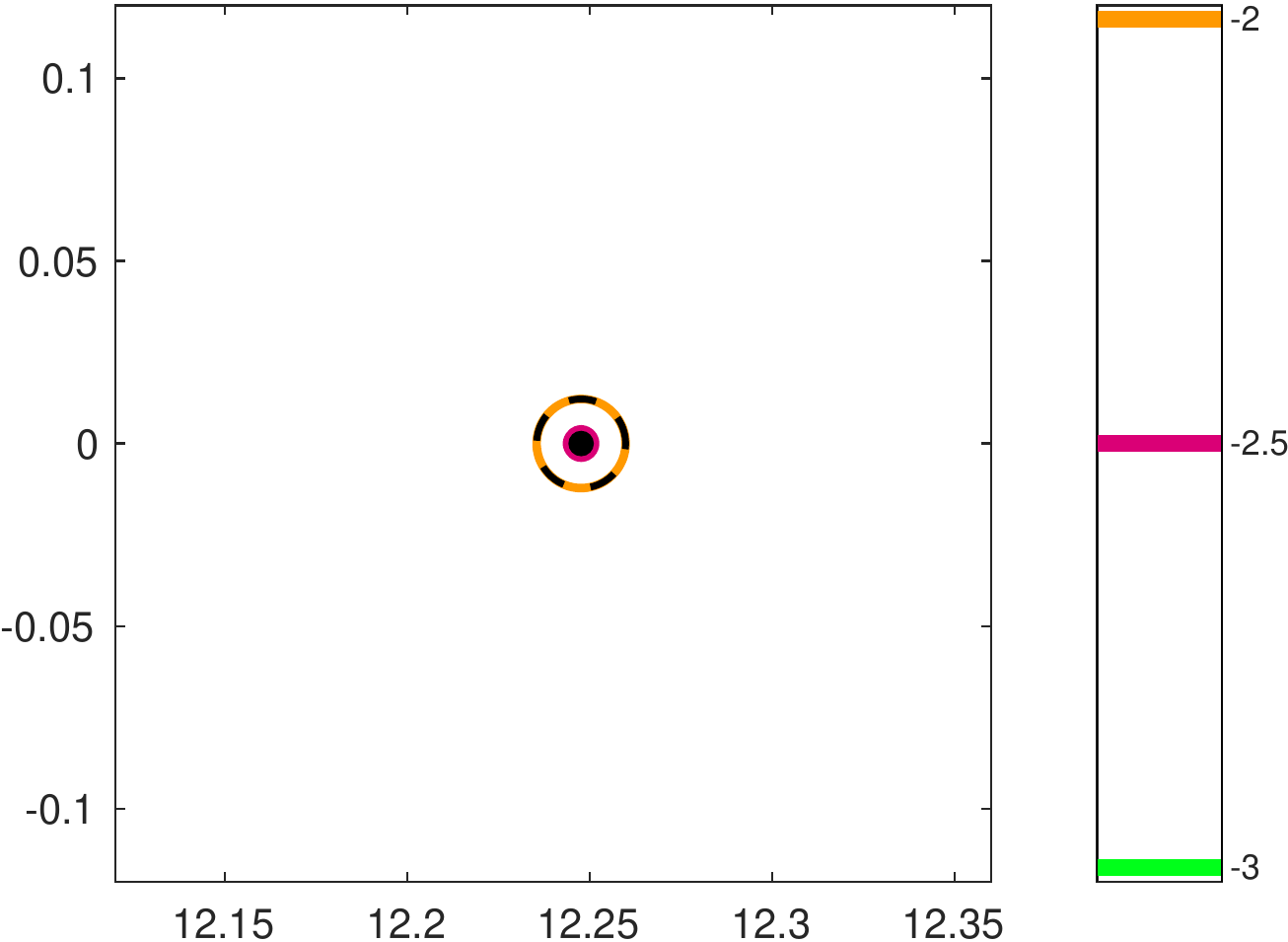}
\end{center}
\caption{Boundaries of $\sigma_\delta(A)$ for $\delta = 10^{-2}, 10^{-2.5}, 10^{-3}$ for an example with $n=25$ and ${\rm rank}({\bf K})=2$.  (The color bar on the right shows $\log_{10}(\delta)$.)
The black dots show the three eigenvalues of ${\bf I} + {\bf K}$.\ \ 
The dashed black curves show the asymptotic sets $\{\lambda_j + \kappa_j \delta  {\rm e}^{{\rm i} \theta}: \theta\in[0,2\pi)\}$ for $\delta=10^{-2}$ .  Even for this relatively large value of $\delta$, the agreement with the boundary of $\sigma_\delta(A)$ is apparent.
\label{fig:Cp1eig}
}
\end{figure}

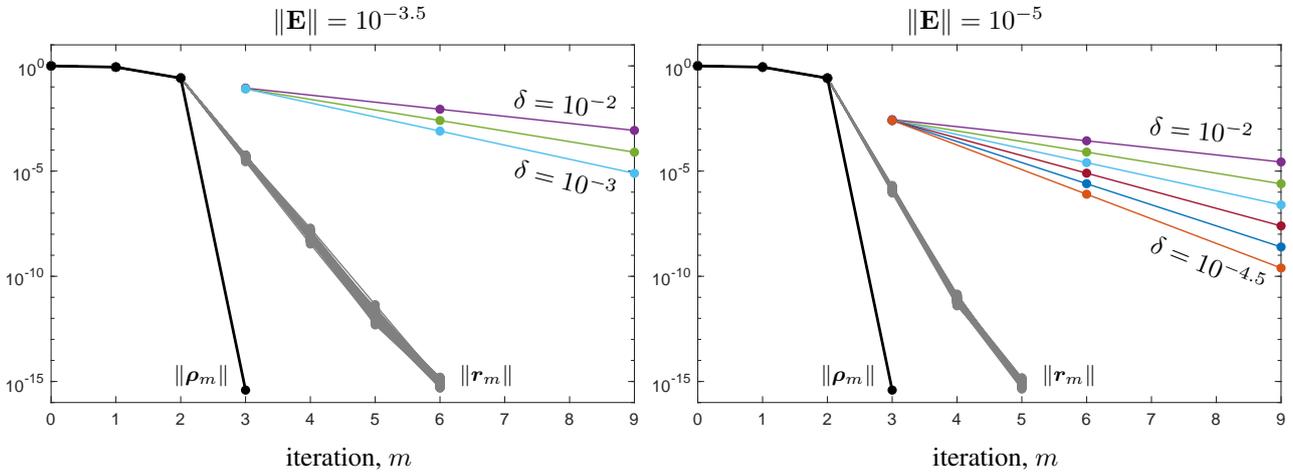
\begin{figure}[t!]
\vspace*{1em}
\begin{center}
\includegraphics[height=2in]{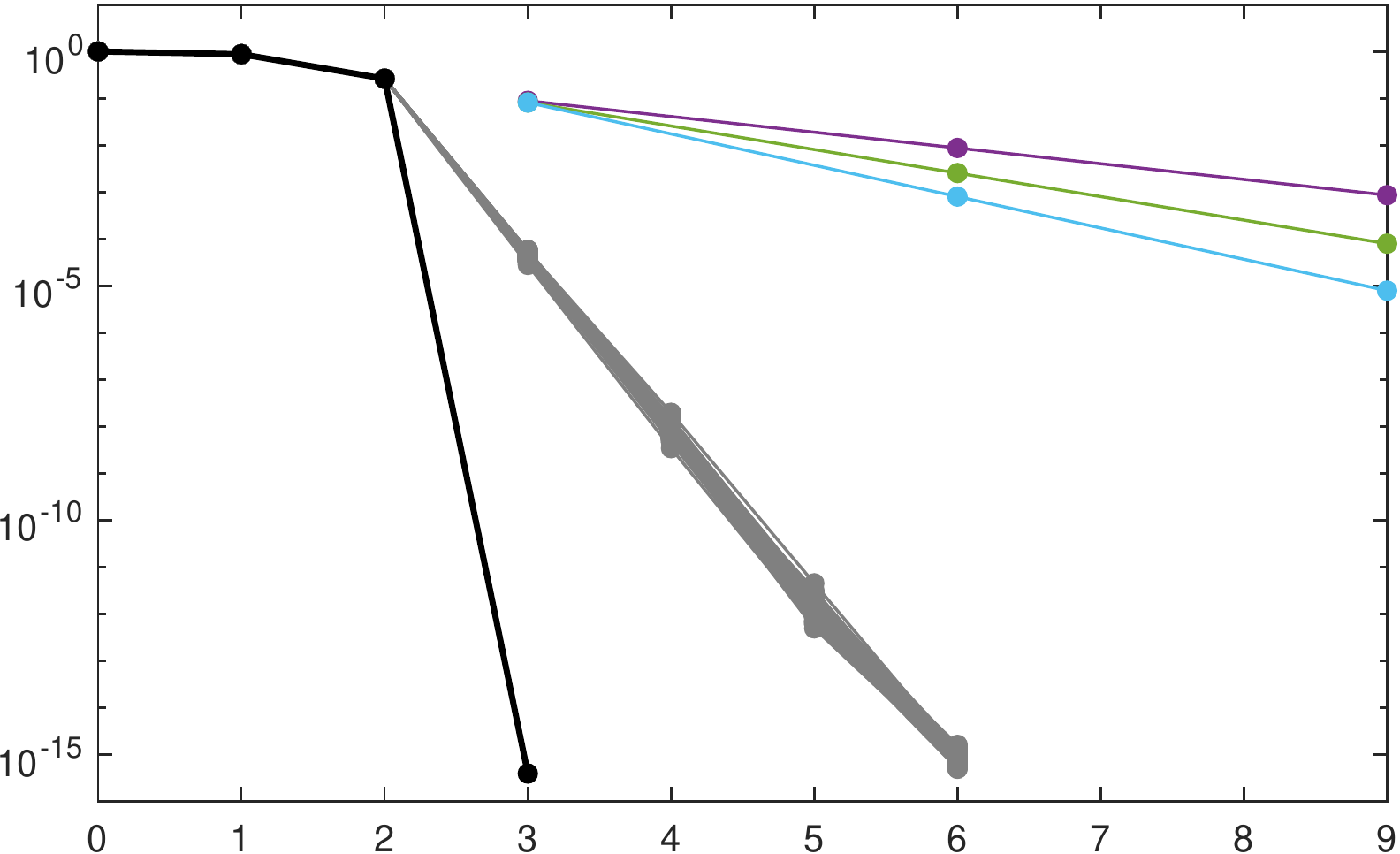}
\begin{picture}(0,0)
\put(-135,-15){iteration, $m$}
\put(-140,150){$\|{\bf E}\|=10^{-3.5}$}
\put(-50, 120){\rotatebox{-6}{$\delta=10^{-2}$}}
\put(-50, 94){\rotatebox{-14}{$\delta=10^{-3}$}}
\put(-177,17){\footnotesize $\|\boldsymbol{\rho}_m\|$}
\put(-70,17){\footnotesize $\|\boldsymbol{r}_m\|$}
\end{picture}
\includegraphics[height=2in]{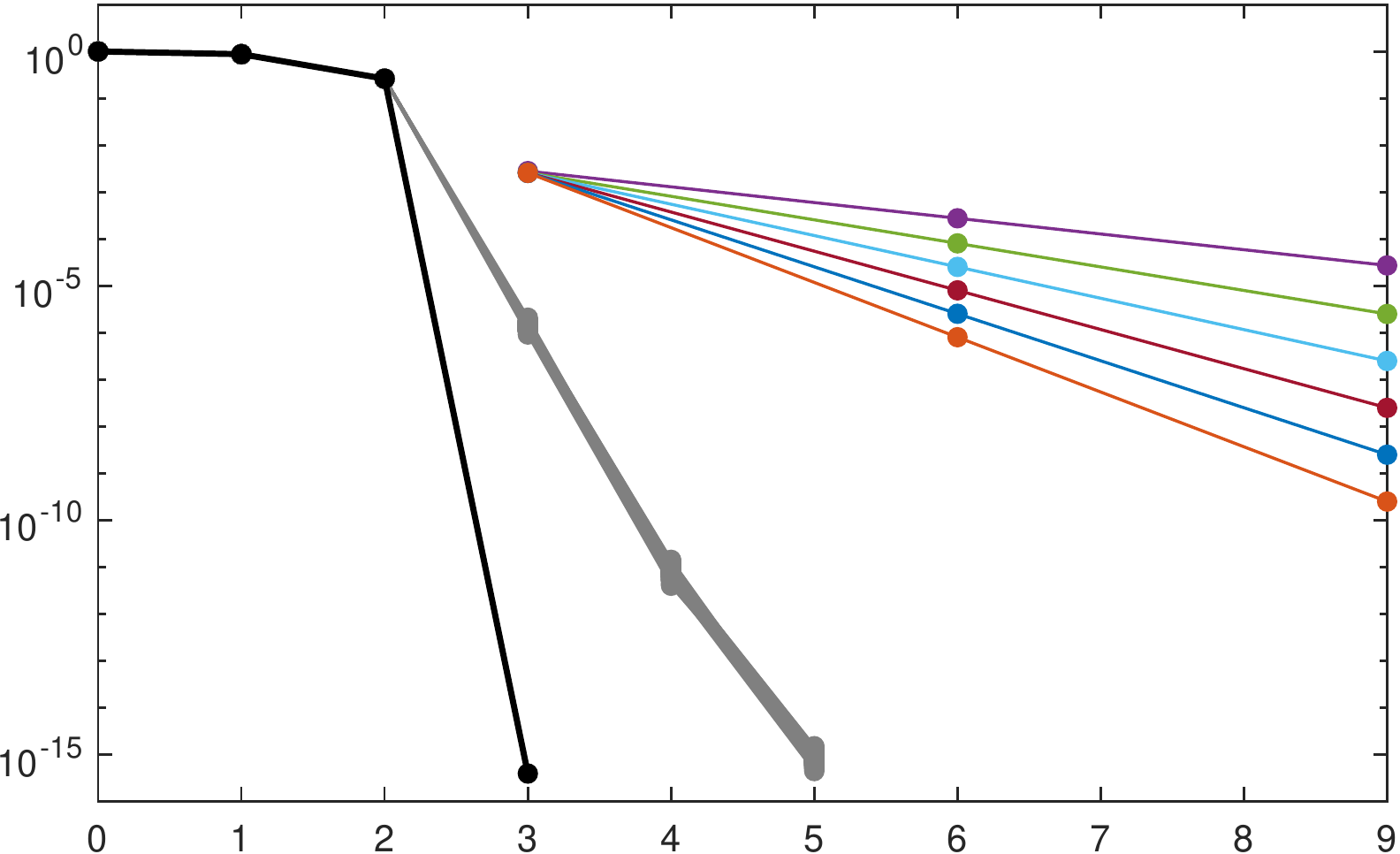}
\begin{picture}(0,0)
\put(-135,-15){iteration, $m$}
\put(-145,150){$\|{\bf E}\|=10^{-5}$}
\put(-54, 110){\rotatebox{-6}{$\delta=10^{-2}$}}
\put(-54, 66){\rotatebox{-22}{$\delta=10^{-4.5}$}}
\put(-177,17){\footnotesize $\|\boldsymbol{\rho}_m\|$}
\put(-94,17){\footnotesize $\|\boldsymbol{r}_m\|$}
\end{picture}
\end{center}

\vspace*{1em}

\caption{GMRES convergence for $\MI+\MK$ (black line) and for $\MI+\MK+\ME$ (gray lines, for 100 random perturbations $\ME$).  The colored lines show the upper bound (\ref{eq:bound}) for iterations $m=3$, $6$, and $9$, for $\log_{10}(\delta) = -2, -2.5, -3$ for $\|\ME\| = 10^{-3.5}$  (left) and 
$\log_{10}(\delta) = -2, -2.5, \ldots, -4.5$ for $\|\ME\|=10^{-5}$ (right).}
\label{fig:Cp1bound}
\end{figure}

Following~\cite{Ymb11} and assuming $\MI+\MK$ is diagonalizable, we can approximate the bounds in terms of the \emph{condition numbers of the eigenvalues} (i.e., the norms of the spectral projectors associated with each distinct eigenvalue; for simple eigenvalues, the secant of the angle between left and right eigenvectors); see, e.g.,~\cite{GV12,TE05}.
Let $\kappa_j\ge 1$ denote the condition number of $\lambda_j$, $j=1,\ldots, p+1$.  
Then the asymptotic behavior of $(z\MI-(\MI+\MK))^{-1}$ as $z\to\lambda_j$ implies that $\sigma_\delta(\MI+\MK)$ behaves, as $\delta\to0$, like the union of $p+1$ disks, each centered at an eigenvalue $\lambda_j$ and having radius $\kappa_j\delta$; see \cite[chap.~52]{TE05}.
(For the example in Figure~\ref{fig:Cp1eig}, we compare these asymptotic sets with the true pseudospectra for $\delta=10^{-2}$.)
For small $\delta$, we can thus estimate $L_\delta = 2\pi\delta(\kappa_1+\cdots+\kappa_{p+1})$ and hence, as $\delta \to 0$,
\[    \2nm{\Vr_{m}} < \epsilon\, C_m(\delta) 
= \epsilon\,\left(\frac{L_\delta\2nm{\Vb}}{\pi\delta^2}\right)\sup_{z\in \sigma_\delta(\MA)}|\psi_m(z)| 
\approx \frac{2\,\epsilon\, (\kappa_1 + \cdots + \kappa_{p+1})}{\delta} \max_{\substack{j=1,\ldots, p+1 \\\theta\in[0,2\pi)}} |\psi_m(\lambda_j + \kappa_j \delta {\rm e}^{{\rm i}\theta})|.\]
Next, we seek more insight into the conditioning of the eigenvalues of $\MI+\MK$.

\section{Sources of Sensitive Eigenvalues}\label{sec:Sources}
While a clustered spectrum away from the origin is often a good
indicator of fast GMRES convergence, convergence may be slow when
some of those eigenvalues are ill-conditioned. In the last section we saw that we can instead focus on the $\delta$-pseudospectrum of $\sigma_{\delta}(\MI+\MK)$.
In fact, there can only be at most $2p$ eigenvalues of $\MI+\MK$ that
are sensitive to the introduction of $\ME$, which we formalize now.  
Let $R(\cdot)$ denote the range of an $n\times n$ matrix and
$N(\cdot)$ denote the null space.

First, consider the (full) singular value decomposition (SVD) of $\MK$ (with $\boldsymbol{\Sigma}_1$ nonsingular), written as
\eqs\label{eq:svdK}
\MK = \left[\begin{array}{c|c} \MU_1 & \MU_2 \end{array} \right]\left[\begin{array}{c|c} {\bf\Sigma}_1 & {\bf 0} \\ \hline {\bf 0} & {\bf 0} \end{array} \right]\left[\begin{array}{c} \MV_1^* \\ \hline \MV_2^* \end{array} \right] = \MU_1{\bf \Sigma}_1\MV_1^*.
\eqe 
Given $\MA$ of the form (\ref{eq:Amat}) and the SVD of $\MK$ as in
(\ref{eq:svdK}), we can study the equivalent eigenvalue problems
\eqs\label{eq:evalA} 
    \MA\Vx = \lambda\Vx \quad\Longleftrightarrow\quad \Vx + \MU_1\boldsymbol{\Sigma}_1\MV_1^*\Vx = \lambda\Vx \quad\Longleftrightarrow\quad \MU_1\boldsymbol{\Sigma}_1\MV_1^*\Vx = (\lambda-1)\Vx.
\eqe
We characterize the eigenpairs $(\lambda,\Vx)$ to identify
the source of  potentially ill-conditioned eigenvalues. We
begin by defining the eigenvectors corresponding to eigenvalue
$\lambda = 1$.
\begin{theorem}\label{teo:null}
$(1,\Vx)$ is an eigenpair of $\MA = \MI+\MU_1\boldsymbol{\Sigma}_1\MV_1^*$ 
if and only if $\Vx \in R(\MV_2)$.
\end{theorem}
\begin{proof}
  Let $(1,\Vx)$ be an eigenpair of $\MA = \MI+\MU_1\boldsymbol{\Sigma}_1\MV_1^*$.  Then,  
  $\Vx+\MU_1\boldsymbol{\Sigma}_1\MV_1^*\Vx = \Vx$ implies $\MU_1\boldsymbol{\Sigma}_1\MV_1^*\Vx = {\bf 0}$, and clearly $\Vx \in N(\MV_1^*) = R(\MV_2)$.\ \ Now assume $\Vzero \ne \Vx \in R(\MV_2) = N(\MV_1^*)$.\ \ Then $\MV_1^*\Vx = {\bf 0}$ and (\ref{eq:evalA}) implies $\l=1$.
\end{proof}

When $\lambda \not = 1$, (\ref{eq:evalA}) implies that the corresponding
eigenvector $\Vx \in R(\MU_1)$. In this case, we need only
analyze the eigenpairs of the small matrix
$\boldsymbol{\Sigma}_1\MV_1^*\MU_1$, which we make precise now.

\begin{theorem}\label{teo:range}
$\Vx = \MU_1\boldsymbol{ \xi}$ is an eigenvector of 
$\MA = \MI+\MU_1\boldsymbol{\Sigma_1}\MV_1^*$ if and only if $\boldsymbol{\xi}$ is an 
eigenvector of $\boldsymbol{\Sigma_1}\MV_1^*\MU_1$.
\end{theorem}
\begin{proof}
  Let $\Vx = \MU_1\boldsymbol{\xi}$ be an eigenvector of $\MI+\MU_1\boldsymbol{\Sigma}_1\MV_1^*$.  
  Then  $\MU_1\boldsymbol{\Sigma}_1\MV_1^*\MU_1\boldsymbol{\xi} = \gamma\MU_1\boldsymbol{\xi}$, with $\gamma = \lambda - 1$. Left multiplying by $\MU_1^*$ gives 
  $\boldsymbol{\Sigma}_1\MV_1^*\MU_1\boldsymbol{\xi} = \gamma\boldsymbol{\xi}$,
  and so $\boldsymbol{\xi}$ is an eigenvector of $\boldsymbol{\Sigma}_1\MV_1^*\MU_1$ with eigenvalue $\gamma = \lambda - 1$.  
  Now assume 
  $\boldsymbol{\Sigma}_1\MV_1^*\MU_1\boldsymbol{\xi} = \gamma\boldsymbol{\xi}$ for $\Vxi\ne \Vzero$.
  Then  
  $\MU_1\boldsymbol{\Sigma}_1\MV_1^*(\MU_1\boldsymbol{\xi}) = \gamma(\MU_1\boldsymbol{\xi})$,
  which is equivalent to (\ref{eq:evalA}) with
    $\lambda - 1 = \gamma$ and $\Vx = \MU_1\boldsymbol{\xi}$.    
\end{proof}

Suppose $\Vxi\ne \Vzero$ is an eigenvector of $\MSigma_1\MV_1^*\MU_1$ corresponding to eigenvalue $\gamma$.
First, consider the case that $\boldsymbol{\Sigma_1}\MV_1^*\MU_1$ is diagonalizable. We need to distinguish between $\gamma = 0$ and $\gamma\not=0$. 
If $\c\neq 0$, $\MU_1\boldsymbol{\xi}$ is an eigenvector of $\MA$ with eigenvalue $\l =\c+1$. 
When $\c=0$, $\l=1$ and $\MSigma_1\MV_1^*\MU_1\Vxi=\Vzero$ implies $\MV_1^*\MU_1\Vxi=\Vzero$, and hence $\MU_1\boldsymbol{\xi} \in R(\MV_2)$ is not a `new' eigenvector of $\MA$. In this case, 
$\MV_1\boldsymbol{\Sigma}_1^{-1}\boldsymbol{\xi}$ is a generalized eigenvector of order two with defective eigenvalue $\l=1$:
$$(\MA - \MI)^2 \MV_1 \boldsymbol{\Sigma}_{1}^{-1}
\boldsymbol{\xi} 
= (\MA-\MI)\MU_1\boldsymbol{\Sigma}_1\MV_1^*
\MV_1\boldsymbol{\Sigma}_1^{-1}\boldsymbol{\xi}
=
(\MA-\MI)(\MU_1\boldsymbol{\xi}) = \Vzero,$$
with $\MU_1\boldsymbol{\xi} \neq \Vo$.
Such defective eigenvalues are sensitive 
to perturbations~\cite{MBO97}. 
Let $\ell$ denote the multiplicity of $\c=0$.
Then we have $\ell$ eigenvectors $\MU_1\boldsymbol{\xi} \in R(\MV_2)$ and $\ell$ generalized eigenvectors of order two as above, $p - \ell$ eigenvectors of type $\MU_1\boldsymbol{\xi}$ (with $\c\neq 0$), in addition to $n-p -\ell$  further eigenvectors in $R(\MV_2)$ for $\l=1$.  

If $\boldsymbol{\Sigma}_1\MV_1^*\MU_1$ is not diagonalizable and $\gamma\not=0$, 
then a generalized eigenvector $\boldsymbol{\xi}$ of order $k$ corresponding to $\gamma$ corresponds to a generalized eigenvector $\MU_1 \boldsymbol{\xi}$ of the same order corresponding to $\lambda = \gamma+1$ of $\MA$. This follows directly from the observation that 
$(\MU_1\boldsymbol{\Sigma}_1\MV_{1}^* - \gamma \MI_n)\MU_1 = \MU_1(\boldsymbol{\Sigma}_1\MV_{1}^*\MU_1 - \gamma \MI_p)$ and hence 
$\MU_1(\boldsymbol{\Sigma}_1\MV_{1}^*\MU_1 - \gamma \MI_p)^{k}\boldsymbol{\xi} = \Vo \Leftrightarrow 
(\MA - (\gamma+1) \MI_n)^{k}\MU_1\boldsymbol{ \xi} = \Vo$.
Again, defective eigenvalues are highly sensitive to 
perturbations~\cite{MBO97}.  

As a result, we get up to $p$ sensitive eigenvalues from small angles between the eigenvectors in $R(\MU_1)$, defective eigenvalues corresponding to $\c=0 \Leftrightarrow \l=1$, or 
$\boldsymbol{\Sigma}_1\MV_1^*\MU_1$
being non-diagonalizable.
Another up to $p$ potentially sensitive eigenvalues arise from small angles between  
eigenvectors 
$\MU_1\boldsymbol{\xi}$ (with $\lambda \not = 1$)
and $R(\MV_2)$, as any vector
$\MV_2\boldsymbol{\zeta}$ is an eigenvector (with $\lambda = 1$), and  from
small angles between eigenvectors 
$\MU_1\boldsymbol{\xi}$ and the vectors 
$\MV_1 \boldsymbol{\Sigma}_1^{-1} \boldsymbol{\xi}$ corresponding to 
eigenvalues $\gamma = 0$. 

If $\boldsymbol{\Sigma}_1 \MV_1^* \MU_1$ is not diagonalizable and $\c=0$, the results above show that 
a Jordan chain of lengths $s$ for $\MSigma_1\MV_1^*\MU_1$ with eigenvalue $\c = 0$, $\boldsymbol{\xi}_1, \ldots, \boldsymbol{\xi}_s$ (the subscript indicating the order) gives a Jordan chain of $\MA$ (with eigenvalue $\l=1$) of length $s+1$ given by 
$\MU_1\boldsymbol{\xi}_1,  
\MV_1\boldsymbol{\Sigma}_1^{-1}\boldsymbol{\xi}_1, \ldots,
\MV_1\boldsymbol{\Sigma}_1^{-1}\boldsymbol{\xi}_s$.

We have identified sources of sensitive eigenvalues for matrices
taking the form (\ref{eq:Amat}).  
Next, we give two small
examples, each highlighting how these sources of sensitivity 
can affect the convergence of GMRES.  

\paragraph*{Example 1.} To visualize the effect of the sensitive eigenvalues of $\boldsymbol{\Sigma_1}\MV_1^*\MU_1$ on GMRES
convergence for matrices of type 
(\ref{eq:Amat}), we consider an example where
$\boldsymbol{\Sigma_1}\MV_1^*\MU_1$ has two ill-conditioned eigenvalues. We let the
matrix size be $n = 25$ and $p={\rm rank}(\MK) = 5$. The sensitive eigenvalues of $\MK$ are
$\gamma_1 \approx 32.62$ and $\gamma_2 \approx -30.62$, with corresponding sensitive eigenvalues of $\MA$, 
$\lambda_1 \approx 33.62$ and $\lambda_2 \approx -29.62$.  We show in
Figure \ref{fig:condNumSmall} that the matrix has two
sensitive eigenvalues (i.e., those shown with large condition number).
Figure \ref{fig:smallEigAll} shows pseudospectra of $\MA$. 
We construct $\ME_1$ and $\ME_2$ such that
$\2nm{\ME_1} = 10^{-3}$ and $\2nm{\ME_2} = 10^{-1}$ and run
GMRES to solve the linear systems $\MA_{1,2}\Vx = \Vb$, for $\MA_1 = \MI+\MK+\ME_1$ and 
$\MA_2 = \MI+\MK+\ME_2$. Here, we let $\Vb$ be a random vector generated by MATLAB's {\tt rand} function. For the unperturbed system, GMRES
converges in 4 iterations. For the system with coefficient matrix $\MA_1$, GMRES converges in 7
iterations, and for $\MA_2$, iterations further increase to 14.
In fact, even for a very small perturbation $\ME_3$ with
$\2nm{\ME_3} = 10^{-6}$, GMRES converges in 6 iterations. 

\begin{figure}[h!]
\centering
\begin{subfigure}{.3\textwidth}
\vspace{12mm}
\centering
\includegraphics[width = 5.5cm]{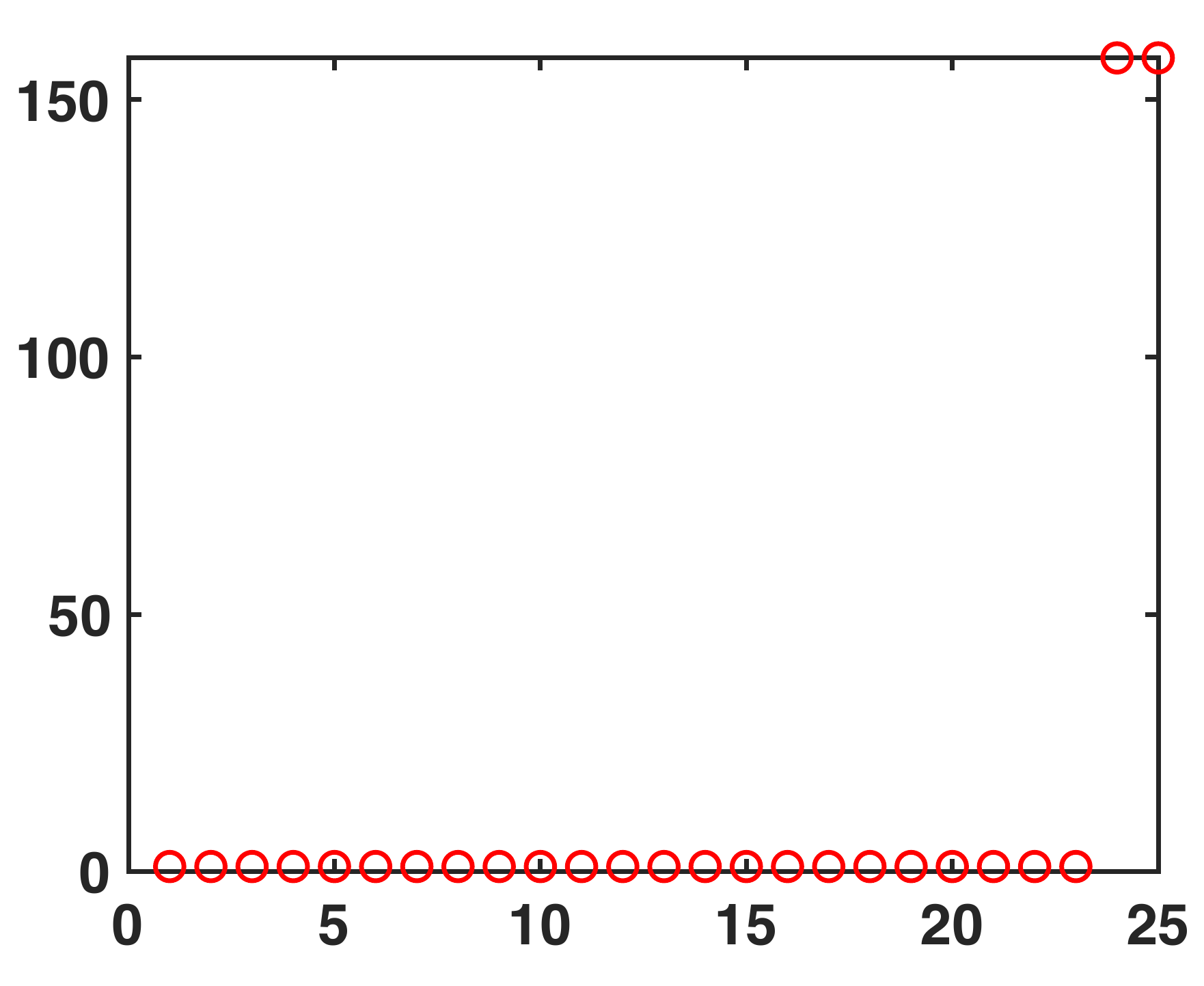}
\centering
\subcaption{Condition number of the eigenvalues.}
\label{fig:condNumSmall}
\end{subfigure}
\hspace{1cm}
\begin{subfigure}{.5\textwidth}
\includegraphics[width = 10cm]{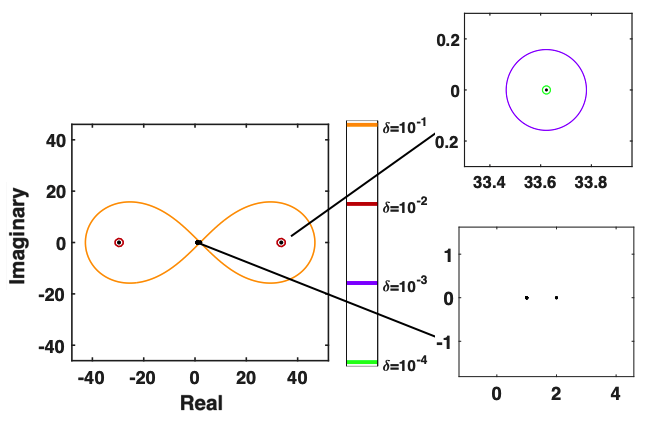}
\subcaption{Pseudospectra of $\MI+\MK$. The color bar gives
$\delta$ corresponding to $\sigma_\delta(\MI+\MK)$.}
\label{fig:smallEigAll}
\end{subfigure}
\caption{$\MI+\MK$ with dimension $n = 25$, rank$(\MK) = 5$, 
and two ill-conditioned eigenvalues of $\boldsymbol{\Sigma}_1\MV_1^*\MU_1$. }
\label{fig:illEig}
\end{figure}

\paragraph*{Example 2.} Next, we analyze a matrix with very small angles between 
$R(\MU_1)$ and 
$R(\MV_2)$. Specifically, we let the matrix size be 
$n = 25$ and let $p={\rm rank}(\MK) = 5$.  We construct $\MK$ such that
there are two small angles, $\theta_{1,2} = 10^{-5}$, between
$R(\MU_1)$ and $R(\MV_2)$.
We see in Figure \ref{fig:condNum} that the unperturbed matrix
we construct has two sensitive eigenvalues. Figure \ref{fig:angles} shows pseudospectra of $\MA$.\ \ We choose $\ME_1$ and $\ME_2$ such that $\2nm{\ME_1} = 10^{-5}$
and $\2nm {\ME_2} = 10^{-2}$ and consider GMRES iterations for
$\MA_1 = \MI+\MK+\ME_1$ and $\MA_2 = \MI+\MK+\ME_2$. For the
unperturbed system, GMRES converges in 6
iterations. For the perturbed system corresponding to
$\MA_1$, iterations increase slightly to 7, and for $\MA_2$
iterations increase to 10.

\begin{figure}[h!]
\begin{subfigure}{.3\textwidth}
\includegraphics[width = 5.6cm]{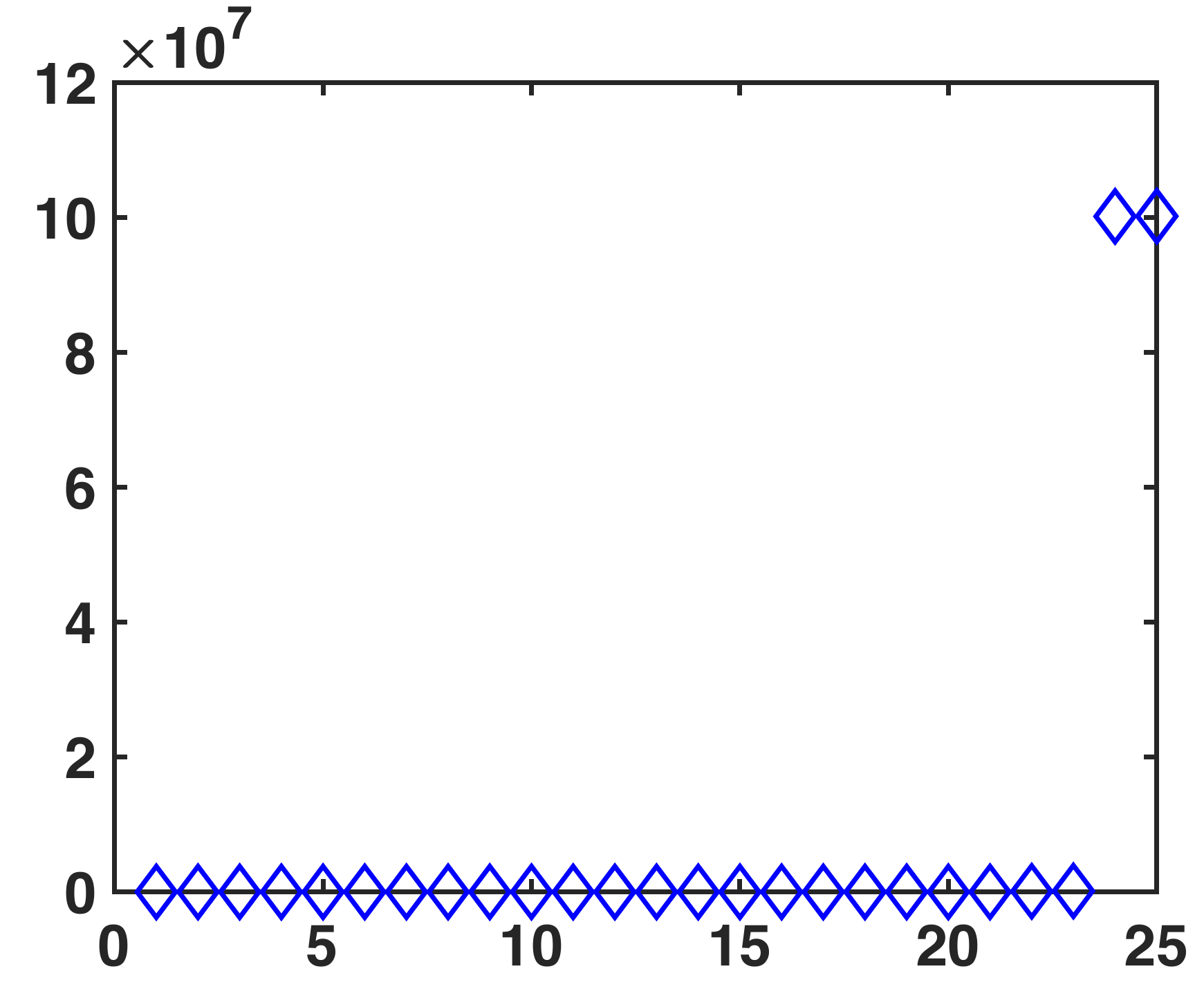}
\vspace{1mm}

\subcaption{Condition number of the eigenvalues.}
\label{fig:condNum}
\end{subfigure}
\hspace{1cm}
\begin{subfigure}{.5\textwidth}
\includegraphics[width = 11cm]{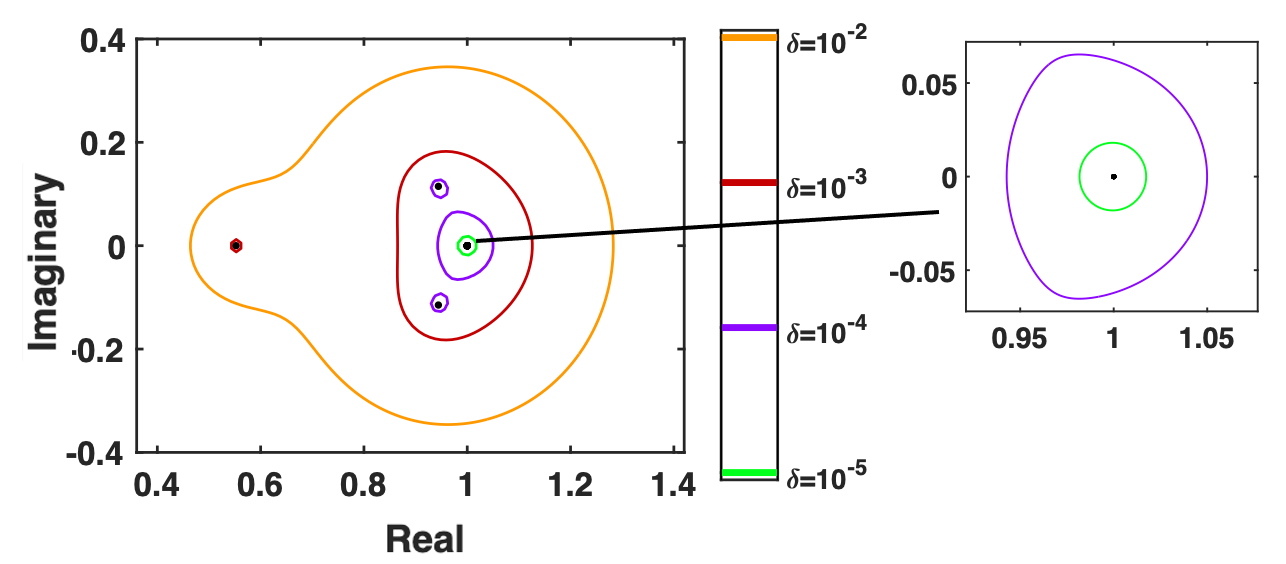}
\subcaption{Pseudospectra of $\MI+\MK$. The color bar gives
$\delta$ corresponding to $\sigma_\delta(\MI+\MK)$.}
\label{fig:angles}
\end{subfigure}
\caption{$\MI+\MK$ with dimension $n = 25$, rank$(\MK) = 5$, 
and two small angles between $\MU_1$ and $\MV_2$.}
\label{fig:closeAngles}
\end{figure}

\section{Experimental Results}\label{sec:results}
We consider the solution of the nonlinear PDE
\eqs \label{eq:nlpde}
  -\nabla\cdot(\nabla u)+(1+u)(70u_x + 70u_y) & = & f  \quad\mbox{on}\quad (0,1)\times (0,1)  ,     \\
  \nonumber
  u & = & 0 \quad \mbox{for} \quad
  x = 0, x = 1 \mbox{ or } y = 0, y = 1 ,
\eqe
on the unit square, where $f$ is chosen so that the solution satisfies
$u(x,y) = y\sin(\p  y)(1-x)\sin(\p x)e^{4x}$. We use a 
finite difference discretization with second order central differences and mesh width 
$h = 1/201$. Discretization leads to the discrete system 
$\MF(\Vu) = \MA \Vu + (\MD(\Vu)+\MI)(70\MD_x + 70\MD_y)\Vu - \Vf = \Vo$, where $\MD(\Vx)$ denotes the diagonal matrix with the coefficients of $\Vx$ on the diagonal, $\MA$ is the discretized diffusion operator, and $\MD_x, \MD_y$ are the convection operators in the $x$ and $y$ directions.
Our initial solution is $\Vu=\Vo$, and hence the initial Jacobian is $\MJ_0 = \MA + \MD((70\MD_x + 70\MD_y)\Vu) + (\MI+\MD(\Vu))(70\MD_x + 70\MD_y)$. 

In general, realistic systems are three dimensional and too large for a direct solver. Thus we consider a very good preconditioner, $\MP$, but not an exact factorization (as, e.g., suggested in \cite{Kelley2003}). The high cost of a very good preconditioner is amortized over multiple linear solves. We use Broyden's method with GMRES for the preconditioned nonlinear system $\MP \MF(\Vu) = 
\Vo$ with initial approximation $\Vu=\Vo$, which results in the initial preconditioned Jacobian $\MB_0 = \MP \MJ_0 = \MI + \ME$,
where $\| \ME \|_2 = \e < 1$.
This leads to a sequence of linear systems of the form $(\MI+\MK+\ME)\Vx = \Vb$, where the rank of $\MK$ increases with each nonlinear iteration. 
We use ILUT \cite{Saad_ILUT94} with drop tolerance $10^{-4}$, which gives $\|\ME\|_2 \approx 0.4$. In practice, we estimate $\|\ME\|_2$ using information from GMRES in the first Broyden step.\ \ 
The line search uses 3-point parabolic interpolation and the Armijo condition \cite{Kelley2003}. Each Broyden step adds a rank-one update. So, at iteration $p$ we solve a system of the form $\MI +\MK + \ME$ with $\mathrm{rank}(\MK)=p$. 

Next, we relate the observed convergence to the analysis provided above, for a realistic $\|\ME\|_2 \approx 0.4$, obtained with an accurate preconditioner. Space limitations prevent us from providing all details, but, for this problem, the analysis shows that the eigenvalues introduced by the low-rank updates are not sensitive and the non-unit eigenvalues $\l$ are very close to the $1+\c_j$ computed from $\boldsymbol{\Sigma}_1\MV_1^*\MU_1$, leading to a marginal increase in the number of GMRES iterations over the Broyden iteration. 
\begin{figure}[!ht]
\begin{center}
  \includegraphics[height=2in]{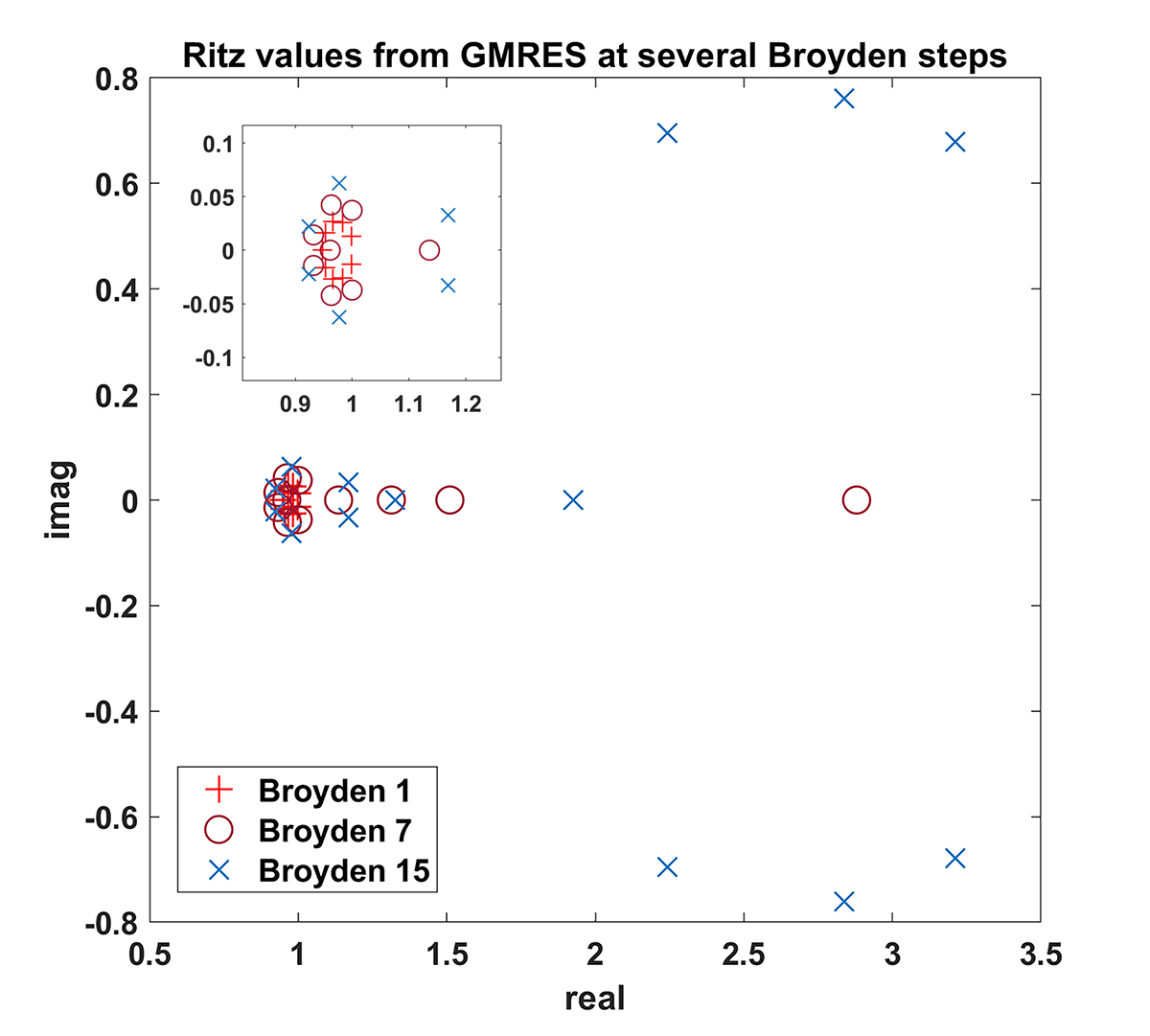}
  \qquad\qquad
  \includegraphics[height=2in]{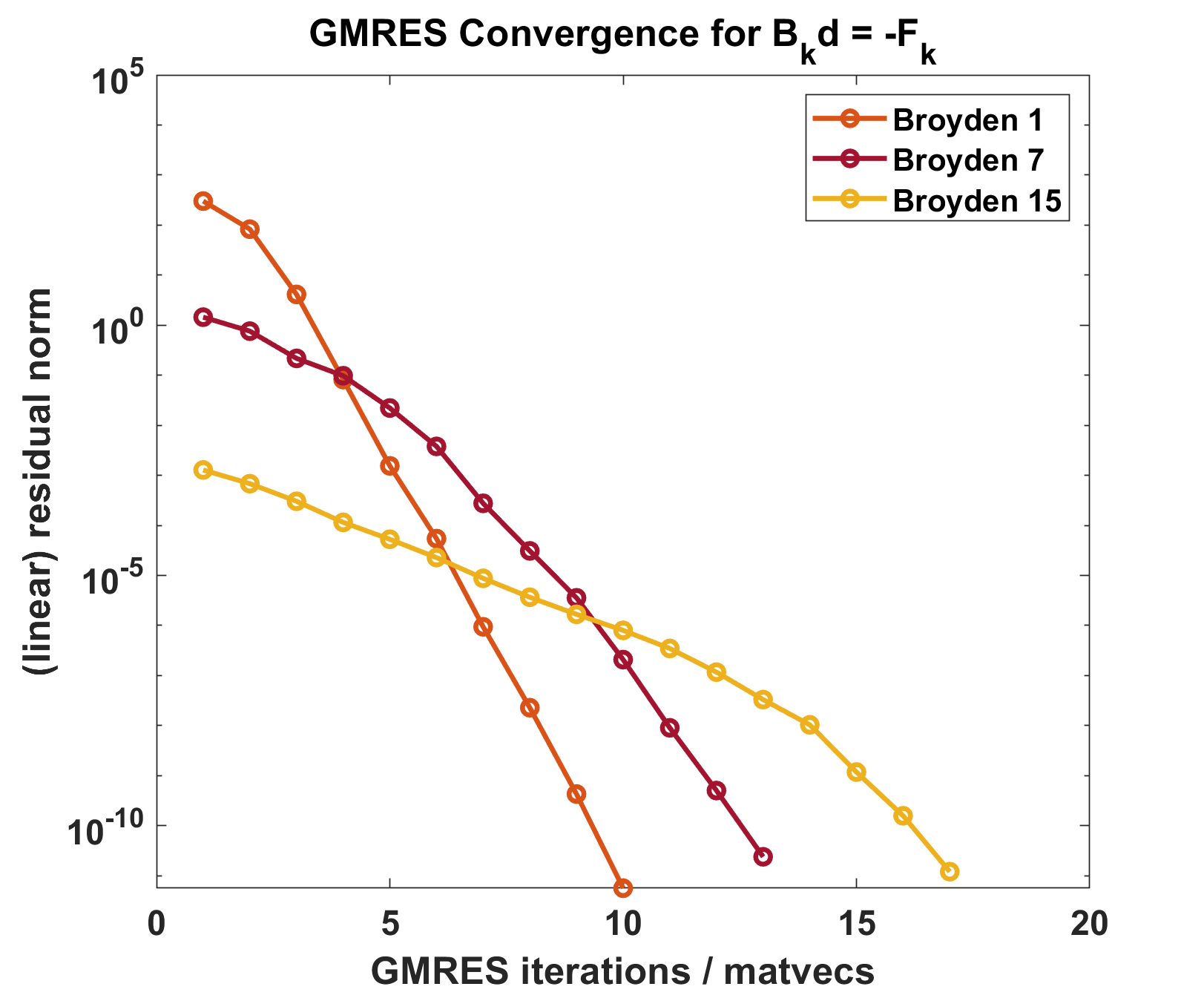}
\end{center}
\caption{
The left figure shows the spectrum of $\MI + \ME$ (first Broyden step) and $\MI + \ME +\MK$ at the 7th and the last Broyden step. The right figure shows the resulting GMRES convergence. Using the analysis above, we could estimate the spectrum and sensitivities from the SVD of $\MK$. }
\end{figure}

\section{Conclusions}\label{sec:conclusion}
In this paper, we characterize GMRES convergence for linear systems with coefficient matrices of the special form $\MI + \MK + \ME$: low rank plus small-in-norm perturbations of the identity matrix. We define a bound for the GMRES residual in terms of the size of $\2nm{\ME}$ by following the theoretical framework of \cite{SifuEmbr13}. We reveal the potentially sensitive eigenvalues of $\MI + \MK$ when introducing small perturbations (i.e., $\ME$), showing that there are no more than $2p$ such eigenvalues, where $p$ denotes the rank of $\MK$.

Several interesting theoretical and practical extensions of the current
paper are the focus of ongoing work, such as considering
$((z-1)\MI-\MK)^{-1}$ directly by way of the Sherman--Morrison--Woodbury formula, and 
when solving long sequences of linear systems of the form 
$(\MI + \MK_j + \ME_j)\Vx_j = \Vb_j$. In particular, we will use this
analysis more extensively to, e.g., determine when to recompute or update
the preconditioner in the Broyden iteration when solving nonlinear PDEs using GMRES.

\vspace{\baselineskip}

\bibliographystyle{pamm}
\bibliography{summaryRef}

\end{document}